\documentclass[11pt]{amsart}

\usepackage{amsmath}
\usepackage{amssymb}
\usepackage{hyperref}
\usepackage{graphics}
\usepackage[english]{babel}
\usepackage{graphicx}
\usepackage{enumitem}
\usepackage[ margin=1.5in]{geometry}

\usepackage{color}
\usepackage{graphics}
\usepackage[english]{babel}
\usepackage{graphicx}

\newtheorem{theorem}{Theorem}[section]
\newtheorem{thm}[theorem]{Theorem}

\newtheorem{lem}[theorem]{Lemma}
\newtheorem{prop}[theorem]{Proposition}

\theoremstyle{definition}
\newtheorem{defn}[theorem]{Definition}

\theoremstyle{remark}

\newcommand{\mbb}{\mathbb}

\newcommand{\NN}{\mbb{N}}

\newcommand{\CC}{\mbb{C}}

\newcommand{\PP}{\mbb{P}}

\newcommand{\FF}{\mbb{F}}

\newcommand{\mc}{\mathcal}

\newcommand{\mcB}{\mc{B}}

\newcommand{\mcH}{\mc{H}}

\newcommand{\mcL}{\mc{L}}
\newcommand{\mcM}{\mc{M}}
\newcommand{\mcN}{\mc{N}}
\newcommand{\mcO}{\mc{O}}

\newcommand{\mcS}{\mc{S}}
\newcommand{\mcT}{\mc{T}}

\newcommand{\mcX}{\mc{X}}

\newcommand{\mfm}{\mathfrak{m}}

\newcommand{\OO}{\mc{O}}

\newcommand{\wht}{\widehat}
\newcommand{\whts}{\wht{s}}

\newcommand{\whtOO}{\wht{\OO}}
\newcommand{\tsigma}{\tilde{\sigma}}

\newcommand{\SP}{\text{Spec }}

\newcommand{\Pic}{\textrm{Pic}}

\newsavebox{\sembox}
\newlength{\semwidth}
\newlength{\boxwidth}

\newsavebox{\semrbox}
\newlength{\semrwidth}
\newlength{\boxrwidth}

\DeclareMathAlphabet{\mathpzc}{OT1}{pzc}{m}{it}
\newcommand{\bb}{\mathpzc{b}}
\newcommand{\wtH}{\widetilde{\mcH}}
\newcommand{\sm}{\textrm{sm}}

\title[Weak approximation over global fields]
{Weak approximation for cubic hypersurfaces and degree $4$ del Pezzo surfaces}

\author{Zhiyu Tian, Letao Zhang}
\address{
CNRS\\Institut Fourier, UMR 5582\\
100 Rue des Math\'ematiques, BP 74\\
38402, Saint-Martin d'H\`eres, France}
\email{zhiyu.tian@ujf-grenoble.fr}
\address{
Department of Mathematics\\ Stony Brook University\\
Stony Brook, NY 11794}
\email{letao.zhang@stonybrook.edu}

\date{\today}

\begin{document}


\begin{abstract}
In this article we prove the following theorems about weak approximation of smooth cubic hypersurfaces and del Pezzo surface of degree $4$ defined over global fields. (1) For cubic { hypersurfaces} defined over global function fields, if there is a rational point, then weak approximation holds at places of good reduction whose residual field has at least $11$ elements. (2)  For del Pezzo surfaces of degree $4$ defined over global function fields, if there is a rational point, then weak approximation holds at places of good reduction whose residual field has at least $13$ elements. (3) Weak approximation holds for cubic hypersurfaces of dimension at least $10$ defined over a global function field of characteristic not equal to $2, 3, 5$ or a purely imaginary number field. 
\end{abstract}


\maketitle

\tableofcontents

\section{Introduction}\label{sec:intro}

Let $K$ be a number field or the function field of a smooth projective geometrically connected curve $B$. Given a variety $X$ defined over $K$, a natural question is to understand the set of $K$-rational points $X(K)$ of $X$. For many types of $K$-varieties $X$, we study the set of rational points $X(K)$ by inspecting its local behavior over the completion $\wht{K}_{\nu}$ at each place $\nu$ of $K$. This naturally leads to the study of local-global principles (i.e. Hasse principle and weak approximation) for $X$.  
 The variety $X$ is said to satisfy {\emph{weak approximation}} if for every finite set of places of $K$ and points of $X$ over these places, given $N\in\NN$ we can find a $K$-rational point of $X$ approximating these points to the $N$-th order.

Smooth cubic hypersurfaces have served as a good testing ground for people to investigate questions related to local-to-global principles. 
For a smooth cubic hypersurface of dimension at least $15$ (or any other smooth hypersurface of sufficiently large dimension) defined over the function field $\FF_q(t)$, weak approximation is established in \cite{leeforms}. Recent work of  \cite{BrowningVisheCubic} shows that weak approximation holds for cubic hypersurfaces of dimension at least $6$ over $\FF_q(t)$ with $q$ not a power of $2$ or $3$. These results depend on the adaption of the circle method to the function field case.

Yong Hu \cite{HuYongWA} studies weak approximation for cubic surfaces defined over global function fields $K=\FF_q(B)$. He proves that if there is a rational point, then weak approximation holds at a single place of good reduction for $q \geq 53$ and the characteristic is not $2$ or $3$, and weak approximation at zero-th order for finitely many places of good reduction for $q\geq 11$. In this article we improve his result as the following.

\begin{thm}\label{thm:good}
Let $X$ be a smooth cubic hypersurface defined over the global function $K=\FF_q(B)$, the function field of a curve $B$ defined over a finite field $\FF_q$. Assume that $X$ has a $K$-rational point. Then $X$ satisfies weak approximation at places of good reduction whose residue fields have at least $11$ elements. In particular, weak approximation at places  of good reduction holds if $q \geq 11$.
\end{thm}
 Here, a place $\nu$ of $K$ is a {\it place of good reduction} for a smooth cubic hypersurface $X$ defined over a global function field $K=\FF_q(B)$ if there exists a smooth projective morphism $\wht{\mcX} \to\SP(\wht{\mcO}_{\nu})$,  
such that the generic fiber is isomorphic to $X \times_K K_{\nu}$ and the closed fiber is a smooth cubic hypersurface, where $\OO_\nu$ is the ring of integers for the local field $K_\nu$.

The same method also applies to weak approximation for a smooth del Pezzo surface of degree $4$.  In this paper, we prove the following. 
\begin{thm}\label{thm:del}
Let $X$ be a smooth del Pezzo surface of degree $4$ defined over $K=\FF_q(B)$, the function field of a curve $B$ defined over a finite field $\FF_q$. Assume that $X$ has a $K$-rational point. Then $X$ satisfies weak approximation at places  of good reduction whose residue fields have at least $13$ elements. In particular, weak approximation at places of good reduction holds if $q \geq13$.
\end{thm}

Here, the definition of places of good reductions is similar to those of cubic hypersurfaces. Namely we require the existence of a  smooth projective morphism $\wht{\mcX}\to\SP \whtOO_{\nu}$  whose central fiber is a smooth del Pezzo surface of degree $4$ (as opposed to any smooth surface).

As a side remark, the anticanonical system gives an embedding of a del Pezzo surface of degree $4$ into $\PP^4$ as the complete intersection of two quadrics. 
Hasse principle and weak approximation for a smooth complete intersection of two quadrics in $\PP^n, n \geq 5$ is recently proved  in  \cite{Hasse}(the argument of weak approximation is a geometric argument due to Colliot-Th\'el\`ene, Sansuc, and Swinnerton-Dyer, \cite{CTSD}, \cite{CTSDquadric}).

Partial results of weak approximation for cubic hypersurfaces in $\PP^n$ over a number field are known either by the descent and fibration methods (\cite{HarariFibration, HarariHasse, CTSalbergerSingularCubic})
or the circle method (\cite{SkinnerWA, SJMWA}). 

In this paper we add the following new result to the list by a geometric argument.

\begin{thm}\label{thm:function}
Let $X$ be a smooth cubic hypersurface of dimension at least $10$ defined over a global field $K$, which is either a purely imaginary number field or $\FF_q(B)$, the function field of a smooth projective geometrically connected curve $B$ defined over a finite field $\FF_q$ whose characteristic is not $2, 3, 5$. Then $X$ satisfies weak approximation.
\end{thm}
\noindent Note that a smooth projective cubic hypersurface of dimension at least $8$ defined over a global field always has a rational point (the number field case is settle by \cite{BrowningVishe}).

The proof of the above theorems falls into two directions. The proof of Theorem \ref{thm:good}  and Theorem \ref{thm:del} is based on the deformation theory of rational curves as developed by Koll\'ar-Miyaoka-Mori \cite{KMM92RC}. The proof of Theorem \ref{thm:function} is on the other hand from geometric manifestation. \\


\textbf{Acknowledgments.} The authors are grateful to Brendan Hassett and Jason Starr for very 
useful comments and discussions.

\section{Preliminaries}\label{sec:pre}

 In this section we will recall some important notions and constructions.


\subsection{Formulation of weak approximation over function fields.}  Let $k$ be a field, either algebraically closed or finite, and let $B$ be a smooth projective curve over $k$, with function field $K = k(B)$. We assume $X$ is a smooth proper variety over $K$. 
\begin{defn}\label{def:model}
A \emph{model} of $X$ is a flat projective morphism $\pi:\mcX \rightarrow B$  with generic fiber isomorphic to $X$.  
\end{defn}
\noindent Each section of the model corresponds to a $K$-rational point of $X$ , and vice versa. 

Now we will unwind the definition of weak approximation for the case where $X$ is a smooth proper variety over $K$, and give the geometric formulation of $X$ satisfies weak approximation,  which is  equivalent to the definition in Section \ref{sec:intro}. For detailed discussion one may find \cite[Section 1]{WASurvey} very helpful. 

Given a sequence $(\nu_1,\ldots,\nu_l)$ of finitely many places of $K$, each $\nu_i$ corresponds to a closed point $b_i\in B$.  Let $\whtOO_{B,b_i}$ denote the completion of the local ring $\OO_{B,b_i}$ at the maximal ideal $\mfm_{B,b_i}$, 
and let $\wht{K}_{b_i}$ be the completion of $K = k(B)$ at $b_i$. Then $X$ satisfies weak approximation over $(\nu_1,\ldots,\nu_l)$  if  the image of  the  diagonal map  $X(K)\rightarrow \prod_{i}X(\wht{K}_{\nu_i})$ is dense, where  $ \prod_{i}X(\wht{K}_{\nu_i})$  takes the product topology of the $\nu$-adic topologies.  Now let us fix a model $\pi:\mcX\to B$,  a place $\nu_i$, and a section $\whts_i$ of the restriction  
\[
\pi|_{\wht{B}_{b_i}}:\mcX\times_B \wht{B}_{b_i}\rightarrow \wht{B}_{b_i}, \wht{B}_{b_i}= \SP \whtOO_{B, b_i}. 
\]
\noindent Note that $\whts_i$ corresponds to a point of  $X(\wht{K}_{b_i})$.

For each $N\in\NN$, basic $\nu_i$-adic open subsets of an open neighborhood of $\whts_i$ in $X(K_{\nu_i})$ consist of 
sections $t$ of $\pi|_{\wht{B}_{b_i}}$, such that $t$  is congruent to $\whts_i$ modulo $\mfm_{B,b_i}^{N+1}$. 
We can reformulate 
weak approximation in the following way.
\begin{defn}\label{defn:wa} Let $X$ be a smooth projective variety over the function $K=k(B)$ of a curve $B$ defined over $k$, where $k$ is either algebraically closed or finite. 
We say $X$ satisfies \emph{weak approximation at order $N$} if there exists a model $\pi:\mcX\to B$  of $X$ such that 
 for  $J$-data:
\[
J =\left(N; \left(b_1,\ldots,b_l \right)_L; \whts_1,\ldots,\whts_l\right),
\]
consisting of a nonnegative integer $N$,  a sequence of distinct closed points $\left(b_1,\ldots,b_l \right)_L$ $\subset B$  such that 
$ \sum \deg [\kappa(b_i):k] = L$,  and a formal power series sections $\left(\whts_1,\ldots,\whts_l\right)$ where each $\whts_i$ is a section of the restriction  
\[
\pi|_{\wht{B}_{b_i}}:\mcX\times_B \wht{B}_{b_i}\rightarrow \wht{B}_{b_i},
\]
\noindent there exists a regular section $\sigma$ of $\pi$ agreeing with $\whts_i$ modulo $\mfm_{B,b_i}^{N+1}$ for each i. We say $X$ satisfies {\emph{weak approximation}} if  $X$ satisfies weak approximation at any order $N\geq 0$.

We say that $X$ satisfies weak approximation at places of good reduction (of order $N$ for some fixed $N\geq 0$) if this holds for any choices of sequence $(b_1,\ldots,b_l)_L$ which correspond to places of good reduction of $X$ over $k(B)$ (for the given $N$).
\end{defn}


\subsection{Iterated blow-ups}\label{sec:iterated}
Let $X$ be a smooth projective variety over  $K=k(B)$ of dimension $n$. 
{Assume we have a model $\pi:\mcX\to B$ of $X$ with a section $\sigma:B\to\mcX^{sm}$,  where $\mcX^{sm}$ is the smooth locus of $\pi$. }

Given the following $J$-data associated with the model
\[
J =\left(N; \left(b_1,\ldots,b_l \right)_L; \whts_1,\ldots,\whts_l\right),
\]
to find a section of the model $\pi$ agreeing with $(\whts_1,\ldots,\whts_l)$ to the $N$-th order is the same as finding a section in  
the $N$-th iterated blowup associated with $J$ (Proposition 11, \cite{HT06}). To be precise, the  {\it iterated blow-up}
\[
\beta(J):\mcX(J)\to\mcX
\]
is obtained by performing  a sequence of blow-ups:  blow up $\mcX$ successively $N$ times, where at each stage the center is the point at which the strict transform of $\whts_i$ meets the fiber over $b_i$. Observe that at each stage we blow up a smooth point of the fiber of the corresponding model and that the result does not depend on the order of the $b_i$. The procedure is depicted  in Section \ref{sec:comb} Figure \ref{fig:teeth}.

The fiber of $\mcX(J)$ over $B$ decomposes into irreducible components
\[
\mcX(J)_{b_i} = E_{i,0}\cup E_{i,1}\cup\dots\cup E_{i,N}
\]
where
\begin{itemize}
	\item $E_{i,0}$ is the strict transform of $\mcX_{b_i}$. 
	\item $E_{i,k}=\textrm{Bl}_{r_{i,k}}(\PP^n)$, $k = 1,\ldots, N-1$, is the blow-up of the intermediate exceptional divisor $\PP^n$ at $r_{i,k}$, the point where the strict transform of $\whts_i$ evaluated at $b_i$ of the $k$th blow-up. Not that we will denote by $\mcX(J_k)$ the $k$-th blow-up with respect to the $J$ data and denote by $\beta(J_k)$ the corresponding blow-up map $\mcX(J_k)\to\mcX$
	\item $E_{i,N}\cong\PP^n$ is the $N$-th exceptional divisor. 
	\item The intersection $E_{i,k}\cap E_{i,k+1}$ is the exceptional divisor $\PP^{n-1}\subset E_{i,k}$, and a strict transform of a hyperplane in $E_{i,k+1}$ for $k=0,\ldots, N-1$. 
	\item  Each $E_{i,k}$ is a $\PP^1$-bundle over the exceptional divisor $\PP^{n-1}$ for $k=1,\ldots, N-1$.
\end{itemize}
\noindent Let $r_i\in E_{i,N}\backslash E_{i,N-1}$ denote the intersection of the strict transform of $\whts_i$ with $E_{i,N}$, and   denote the morphism of $k$-the blow-up over $B$ as
\[
\beta(J_k): \mcX(J_k)\to B.
\]
For each section $\sigma'$ of $\beta(J_k):\mcX(J_k)\to B$, the composition $\beta(J_k)\circ \sigma'$ is a section of $\pi:\mcX\to B$. Conversely, given a section $\sigma'$ of $\beta(J_k)$, its strict transform in $\mcX(J_k)$ is the unique section of $\mcX(J_k)\to B$ lifting $\sigma'$. In particular, given  a section $\sigma':\mcX(J)\to B$ with $\sigma'(b_i)=r_i$, that is the section intersects the strict transform of $\whts_i$, yields a section $\sigma$ of $\mcX\to B$ where $\sigma\equiv \whts_i\mod \mfm^{k+1}_{B, b_i}$, for $i = 1,\ldots,l$ and for $k = 0,\ldots, N$.


\section{Proof of Theorem \ref{thm:good}}\label{sec:cubic}
One of the main ingredients in our proof is the deformation technique developed by Koll\'ar-Miyaoka-Mori \cite{KMM92RC},
which was later used by Hassett-Tschinkel \cite{HT06} to prove weak approximation at places of good reduction of rationally connected varieties defined over $\CC(B)$, 
the function field of a complex curve. 
Later Yong Hu \cite{HuYongWA} applied this method to prove weak approximation of order $0$ at places of good reduction for cubic surfaces over global function fields.
Our approach is motivated by their works. 
The key to prove weak approximation of an arbitrary order is to show uniform boundedness results for the construction of combs, which we prove in the following section.


\subsection{Construction of bounded family of combs.}\label{sec:comb} 
\begin{defn}\label{defn:comb}
Let $k$ be a field. A {\it comb with $n$ teeth} over $k$ is a nodal curve $T$ over $k$, where $T$ is a union of  $m+1$ curves:  a curve $D$ and the union of $n$ curves  $T_1\cup\dots\cup T_n$ satisfy the following conditions:
\begin{itemize}
\item $D$ is a smooth and geometrically irreducible curve, defined over $k$.
\item Each $\bar{T}_i:=T_i\otimes_k \bar{k} $ is  a chain of $\PP_{\bar{k}}^1$'s. 
\item The union $T_1\cup\dots\cup T_n$ is defined over $k$; however individual curve may not be defined over $k$.
\item Each $\bar{T}_i$ meets $\bar{D}:=D\times_k\bar{k}$ transversally in a single smooth point of $\bar{D}$; however the point may not be defined over $k$.
\item $\bar{T}_i\cap \bar{T}_j = \emptyset$,  for all $i\neq j$.
\end{itemize}
\end{defn}
\noindent Here the curve $D$ is called the {\it handle} of the comb and each $T_i$ a {\it tooth}.

The construction of comb is crucial to produce sections passing through some given points. The following lemmas will produce bounded family of sections over an algebraically closed field (cf. Lemma \ref{lem:preind1}) and then over a finite field (cf. Lemma \ref{lem:preind2}). 

\begin{lem}\label{lem:preind1} 
Let $\pi:\mcX\to B$ be a flat projective family defined over an algebraically closed field $k = \bar{k}$, such that 
 the generic fiber is smooth projective and separably rationally connected. 
Let $\mcH$ be an ample divisor on $\mcX$.
For any triple of positive integers $\left(d,L,N\right)$,
there exists an integer $r:=r(d, L, N)>0$ such that for any section $\sigma:B\to\mcX^\sm$ in the smooth locus of $\pi$,  $\deg_\mcH\left(\sigma\right)\leq d$ and any  sequence of distinct closed points $\left(b_1,\ldots,b_L\right)\subset B$, 
 we can construct a comb $T$  together with a morphism $f:T\to  \mcX$ which satisfies the following conditions. 
\begin{enumerate}
	\item The comb $T$ takes  $B$ as its handle and the morphism $f$ restricted to $B$ is the section $\sigma: B \to \mcX$.
	\item The comb $T$ has $r$-teeth mapping to $r$ very free curves
$\left\{T_1,\ldots,T_{r}\right\}$ with bounded $\mcH$-degree in general fibers outside $\left(b_1,\ldots,b_L\right)$.
      \item The morphism $f: T \to \mcX$ is an immersion. In particular, the normal sheaf $\mcN_T$, defined as the dual of the kernal of the surjection $f^*\Omega_\mcX \to \Omega_T$, is locally free.
	\item If the fiber dimension is at lease $3$, then we may choose the  immersion $f: T \to \mcX$  to be an embedding.
	\item $H^1\left(B, \mcN_{T}|_B\left(-\left(N+1\right)\left( \sum b_i \right)\right)\right) = 0$. 
	\item The sheaf $\mcN_{T}|_B\left(-\left(N+1\right)\left( \sum b_i \right)\right)$ is globally generated.
\end{enumerate}
\end{lem}

\begin{proof}
In our proof, all degrees are taken with respect to the ample divisor $\mcH$.
In general, by attaching sufficiently many very free curves to $\sigma(B)$  in general fibers outside $\left(b_1,\ldots,b_L\right)$, one can always obtain a comb that satisfies all conditions in our Lemma. Note that the existence of very free curves is guaranteed by the separable rational connectedness of general fibers of $\pi$.  

Furthermore, the fixed number $r$ of very free curves one needs to attach so that the conditions are satisfied for \emph{all} choices of $b_1, \ldots, b_L$ follows from an upper-semi-continuity argument. 
Namely, there exists an integer $r>0$, such that for any section $\sigma: B\to\mcX$ of degree at most $d$ and for any line bundle $\mcL\in\Pic^{(N+1)\cdot L}(B)$, we can construct a comb $T$ with handle $\sigma(B)$, and teeth $r$ very free curves $\left\{T_1,\ldots,T_{r}\right\}$ of bounded degree attached to $\sigma(B)$ in general fibers along general directions such that
\[
H^1(B, \mcN_{T}|_B\otimes\mcO(-\mcL)) = 0,
\]
and the normal bundle $\mcN_{T}|_B\left(-\mcL\right)$ is globally generated.
Furthermore we may choose the morphism to be an immersion (or an embedding if the fiber dimension is at least $3$).

So given a sequence of distinct closed points $(b_1,\ldots,b_L)\subset B $, we may choose general $\left\{T_1,\ldots,T_{r}\right\}$ away from the fibers over the points $b_i$, such that
\[
H^1\left(B, \mcN_{T}|_B\left(-\left(N+1\right)\left( \sum b_i \right)\right)\right) = 0, 
\]
and the normal bundle $\mcN_{T}|_B\left(-\left(N+1\right)\left( \sum_{i}b_i \right)\right)$ is globally generated.
We have lots of freedom in choosing the very free curves $T_i$. In particular we can always choose the $T_i$ to be immersions,  to have bounded degree, and to intersect the section $\sigma$ transversely.
\end{proof}

Then we can use the Lang-Weil estimate to prove the following boundedness result over finite fields.

\begin{lem}\label{lem:preind2} 
Let $\pi:\mcX\to B$ be a flat projective family over a finite field $\FF_q$,  such that 
 the generic fiber is smooth projective and separably rationally connected. 
Let $\mcH$ be an ample divisor on $\mcX$.
For any triple of positive integers $\left(d,L,N\right)$,
there exist integers $r>0$,  $m_0> 0$, and $d>0$ 
such that for any $m\geq m_0$, any section $\sigma:B\to\mcX^\sm$ defined over $\FF_{q^m}$, $\deg_\mcH\left(\sigma\right)\leq d$ 
and any sequence of distinct closed points $\left(b_1,\ldots,b_l\right)_L\subset B$ defined over $\FF_{q^m}$, 
where $ \sum  \deg [\kappa(b_i):\FF_q]\leq L$, 
we can construct a comb $T$  together with a morphism { $f:T\to \mcX$}
satisfying the following conditions. 
\begin{enumerate}
	\item The comb $T$ takes  $B$ as its handle and the morphism $f$ restricted to $B$ is the section $\sigma: B \to \mcX^\sm$.
	\item The comb $T$ has $r$ teeth mapping to $r$ very free curves
$\left\{T_1,\ldots,T_{r}\right\}$ with bounded $\mcH$-degree in general fibers outside $\left(b_1,\ldots,b_l\right)$.
\item The morphism $f: T \to \mcX$ is an immersion. In particular, the normal sheaf $\mcN_T$, defined as the dual of the kernal of the surjection $f^*\Omega_\mcX \to \Omega_T$, is locally free.
	\item If the fiber dimension is at lease $3$, then we may choose the  immersion $f: T \to \mcX$  to be an embedding.
	\item $H^1\left(B, \mcN_{T}|_B\left(-\left(N+1\right)\left( \sum b_i \right)\right)\right) = 0$. 
	\item The sheaf $\mcN_{T}|_B\left(-\left(N+1\right)\left( \sum b_i \right)\right)$ is globally generated.
\end{enumerate}

\end{lem}
\begin{proof} In our proof, all degrees are taken with respect to the ample divisor $\mcH$. 
Let $\pi:\mcX\to B$ be a model over a finite field $\FF_q$, let $S$ denote places of bad reduction, and let $(d,L,N)$ be a triple of positive integers. Let $r$ be the number of very free curves obtained from Lemma \ref{lem:preind1} for the model $\pi:\mcX\to B$ over $\bar{\FF}_q$  and the triple  $(d,L,N)$. 

Consider the space $\mcS$ parameterizing the following data: 
\[
\left\{
\begin{array}{l}      
	\textrm{a section }{\sigma:B\to\mcX^\sm}, \deg(\sigma)\leq d;\\
	\textrm{a collection of }L\textrm{ distinct points }  (b_1,\ldots,b_L)\subset \underbrace{B \times \ldots \times B}_{L};\\
	\textrm{a collection of }r\textrm{ distinct points }  \{x_1,\ldots,x_r\}\in \underbrace{B \times \ldots \times B}_{r};\\
	\textrm{relative tangent directions } (v_1, \ldots, v_r) \in \left( (\mcT_{\mcX}^{\text{rel}})|_{\sigma(x_1)},\ldots, (\mcT_{\mcX}^{\text{rel}})|_{\sigma(x_r)}\right);\\
\end{array}
\right\}
\]
subject to the following requirements:
\begin{itemize}
\item
The fiber $\mcX_{b_i}$ over each $b_i$ is smooth and separably rationally connected.
\item
The two sets of points $\{x_j\}_{j=1}^{r}$ and $\{b_i\}_{i=1}^{L}$ has empty intersection.
\item
One can attach $r$ very free curves of bounded degree to $\sigma(B)$ at $\{x_j\}_{j=1}^r$ along the tangent directions $(v_1, \ldots, v_r)$ to obtain a comb $T$, such that 
\[
H^1\left(B, \mcN_{T}|_B\left(-\left(N+1\right)\left( \sum b_i \right)\right)\right)=0
\]
and  $\mcN_{T}|_B\left(-\left(N+1\right)\left( \sum b_i \right)\right)$ is globally generated. 
\end{itemize}

Let $\mathcal{B}$ be the space parameterizing the pair consisting of a section $\sigma$ of degree at most $d$ and a collection of $L$ distinct points $(b_1, \ldots, b_L)$ in $B\backslash S$, where fibers $\mcX_{b_i}$  are smooth and separably rationally connected.
 Lemma \ref{lem:preind1} ensures that the fiber of $\mcS$ over any point $(\sigma, (b_1, \ldots, b_L)) \in \mathcal{B}$ is non-empty.
Note that the fiber of $\mcS$ over any point $(\sigma, (b_1, \ldots, b_L)) \in \mathcal{B}$ is geometrically irreducible.

Now we can complete the proof by applying the Lang-Weil estimate \cite{LW54}. 
For a geometrically irreducible quasiprojective variety $U$, we take its closure $\bar{U}$ in some $\PP^n$. 
Denote by $\partial U$ the complement of $U$ in $\bar{U}$. 
In order to apply the Lang-Weil estimate to find a rational point in $U$ for every large enough finite field,
we need to bound the numbers $\deg \bar{U}$, $\deg \partial U$, $\dim U$, $n$.
If we have a family of geometrically irreducible varieties $p: \mathcal{U} \to \mathcal{T}$, with $\mathcal{U}$ and $\mathcal{T}$ quasi-projective,
then we can realize the morphism $p$ as an open immersion into $\overline{\mathcal{U}}$ which is projective over $\mathcal{T}$. 
Thus there is a universal bound of the above numbers for each fiber over $\mathcal{T}$.
We apply this to the family $\mathcal{S}$ over the base $\mathcal{B}$.
Therefore, there exists an integer $m_0'$ such that for any $m\geq m_0'$, 
for any section $\sigma:B\to\mcX^\sm$ and $(b_1, \ldots, b_l)_L\subset B$ defined over $\FF_{q^m}$, 
 there is a $\FF_{q^m}$-rational point of the fiber of $\mcS$ over 
$\left(\sigma,(b_1, \ldots, b_l)_L\right)$.

By \cite[Theorem 2]{KollarSzabo}, there exist integers $a_0>0$ and $d_0> 0$ such that for any $m\geq a_0$,  $b\in B\backslash S$, any $\kappa(b)$-point $x$ in $\mcX_b$, any tangent vector $v$ in the relative tangent direction $\mathcal{T}_x^{\text{rel}}$, there exists an immersed very free rational curve $C\in \mcX_b$ of degree at most $d_0$ passing through $x$ over $\FF_{q^m}$ whose tangent direction at $x$ is $v$.

We take $m_0$ to be the larger of $m_0'$ and $a_0$.
\end{proof}

\subsection{ The induction and the reduction.} 
Section \ref{sec:comb} provides  a construction of combs with bounded degree, and these combs will be used   to construct sections for weak approximation of a prescribed order in this section. 
 We will adapt notations from Section \ref{sec:pre}.

\begin{lem}\label{lem:key}
Let $\pi:\mcX\to B$ be a flat projective family over a finite field $\FF_q$.   Assume that 
{ there is a section $\sigma:B\to\mcX^{sm}$ in the smooth locus of $\pi$}, and the generic fiber is smooth projective and separably rationally connected. 
Let $\mcH$ be an ample divisor on $\mcX$.
Let $S\subset B$ denote the finite set of points, 
the fibers over which are either singular or not separably rationally connected. 
Then for any two positive integers $N$ and $L$, there exist two integers $C_{N, L}, d_{N, L}$, 
both of which only depend on the numbers $N$ and $L$, 
and a bounded family  $\mcM_{N,L}$ parameterizing sections  in the smooth locus of $\pi$ with $\mcH$-degree at most $d_{N,L}$, 
such that the following is true:

 For every $m \geq C_{N,L}$, every collection of closed points
$(b_1,\ldots,b_l)_L\subset B\backslash S$ defined over $\FF_{q^m}$, with the property that $\sum_{i} \deg [\kappa(b_i):\FF_{q^m}]=L$,  
 and every sequence $(\whts_1,\ldots,\whts_l)$ of formal power series sections of $\pi$  over $\SP \whtOO_{B, b_1},\ldots, \SP \whtOO_{B, b_l}$, 
there is {a section $s: B\to\mcX^{sm}$}, parameterized by an $\FF_{q^m}$-point of $\mcM_{N,L}$ such that the section $s$ of $\pi$ is congruent to $\whts_i$ modulo $\mfm_{B,b_i}^{N+1}$ for  $i=1,\ldots, l$.
\end{lem}

\begin{proof}
Without specification, all degrees discussed in our proof are taken with respect to $\mcH$. We use induction on $N$ and take $m_0(N, L)$ to be the $m_0$ given in Lemma \ref{lem:preind2}. We also adopt notations of iterated blow-ups from Section \ref{sec:iterated}. 

The case where $N=0$ follows from Proposition 3 \cite{HuYongWA}; that is there exists a lower bound $C_{0,L}$,  an upper bound $d_{0,L}$, and a bounded family $\mcM_{0,L}$ and  of sections $\sigma_0: B\to \mcX^\sm$ of degree at most $d_{0,L}$ , such that for every $m\geq C_{0,L}$, weak approximation at order $N=0$ holds at any finite sequence of distinct closed points $(b_1,\ldots,b_l)_L\subset B\backslash S$ defined over $\FF_{q^m}$. 

Assume the statement is true for $n=N-1$.

\begin{figure}[!ht]
 \centering
\vspace*{-0.5in}
\hspace*{.1in}
 \includegraphics[width=6in]{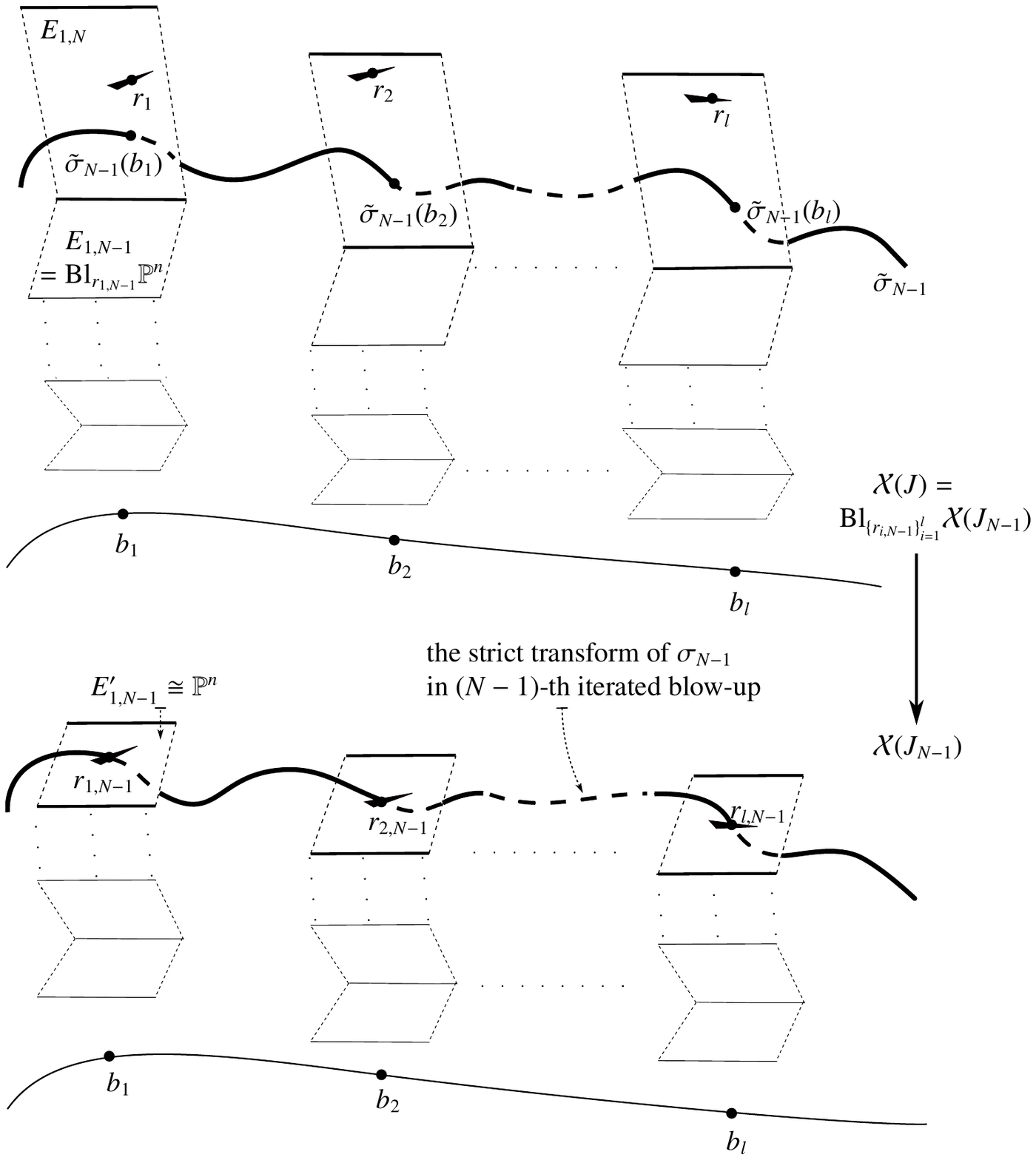}
\vspace*{-3in}
 \caption{From $(N-1)$-th iterated blow-up to the $N$-th iterated blow-up} 
 \label{fig:comb} 
\end{figure}

For $n=N$ and  $m\geq C_{N-1, L}$  , we consider the following $J$-data,
\[
J =\left(N; \left(b_1,\ldots,b_l \right)_L; \whts_1,\ldots,\whts_l\right), 
\]
where  
$\left(b_1,\ldots,b_l \right)_L\subset B\backslash S,  
\sum_{i} \deg [\kappa(b_i):\FF_{q^m}]=L$, and $(\whts_1,\ldots,\whts_l)$ is a formal power series sections of $\pi$ over $\SP \whtOO_{B, b_i}$  over $\FF_{q^m}$. The same as Section \ref{sec:iterated}, we consider the iterated blow-up
\[
\beta(J):\mcX(J)\to\mcX(J_{N-1})\to\ldots\to \mcX
\]
associated with the $J$-data.

By the assumption of  induction, we have  a bounded family $\mcM_{N-1,L}$ of sections {$\sigma_{N-1}:B\to\mcX^\sm$} of  degree  at most  $d_{N-1,L}$, such that
\[
\sigma_{N-1} \equiv \whts_i\mod\mfm_{B,b_i}^{N-1+1} \textrm{ for }i = 1,\ldots, l.
\]
In terms of iterated blow-up showing in Figure \ref{fig:comb}, this is equivalent to the following: in the $(N-1)$-th iterated blow-up $\mcX(J_{N-1})\to\mcX$,  the strict transform of 
${\sigma}_{N-1}$  meets the strict transforms of $\whts_i$ in exceptional divisors $E_{i, N-1}'\cong\PP^n$ , for $i=1,\ldots,l$. We denote these intersection points by $r_{i,N-1}$  for $i=1,\ldots,l$. Then $\mcX(J)$ is obtained from 
  blowing up $\mcX(J_{N-1})$ at  $r_{i,N-1}\in E_{i,N-1}'$ for $i=1,\ldots,l$.

The fiber $\mcX(J)_{b_i}$ decomposes into irreducible components
$$\mcX(J)_{b_i} = E_{i,0 } \cup E_{i,1} \cup \ldots \cup E_{i,N-1}\cup E_{i,N}.$$
Let $r_i\in E_{i,N}\backslash E_{i, N-1}$ be the intersection of the strict transform of $\whts_i$ and $ E_{i,N}$, 
let  $\tilde{\sigma}_{N-1}$ denote the strict transform of $\sigma_{N-1}$ in $\mcX(J)$,
 and let $\wtH\subset\mcX(J)$ denote the pullback of the ample divisor $\mcH$.

Over $\FF_{q^m}, m \geq \max({C_{N-1, L}, m_0})$, 
 we can construct a comb $T_{J}$ with an immersion $f_J:T_{J}\to \mcX(J)$ as follows.  
 Figure \ref{fig:teeth} depicts the construction of a teeth of $T_J$. 
\begin{enumerate} [label= Step \arabic*:]
	\item  Take $B$ as the handle and $f$ restricted to $B$ is the section $\tsigma_{N-1}(B)$.
	\item  For the pair of points  $r_i$ and $\tsigma_{N-1}(b_i)$, connect them with a line $L_{i,N}\subset E_{i,N}\cong\PP^n$, and  $L_{i,N}$ intersects the exceptional divisor $\PP^{n-1}\subset E_{i, N-1}$ at a point $P_{i, N-1}$ .
	\item Since $E_{i, N-1}$ is a $\PP^1$-bundle  over  $\PP^{n-1} = E_{i,N-1}\cap E_{i, N}$, we obtain a $\PP^1$ passing  $P_{i, N-1}$  and intersects the exceptional divisor  $\PP^{n-1}\subset E_{i, N-2}$ at a point $P_{i,N-2}$.
	\item Repeat the previous step till $E_{i,1}$, we obtain a chain of $\PP^1$'s starting from $E_{i,N-1}$ and end with $E_{i,1}$. Let $P_i$ be the intersection between $\PP^1\subset E_{i,1}$ and the exceptional divisor of $E_{i,0}$, the proper transform of $\mcX_{b_i}$.
	\item Attach a very free  curve $C_i$  at $P_i$ in $E_{i,0}$, where $C_i$ meets the exceptional divisor $\PP^{n-1}$  of $ E_{i,1}$ transversely at $P_i$ and   $\deg_{\wtH}(C_i)\leq d_0$. 
	\item Take a chain of $N+1$ $\PP^1$'s as a teeth $T_i$ attached to $B$ at $b_i$. Map the teeth $T_i$ to the union of $L_{i,N}$, $\PP^1$'s, and $C_i$ obtained from the above steps. We have the bound $\deg_{\wtH}(T_i)\leq d_0$.
	  This can be achieved by \cite[Theorem 2]{KollarSzabo}. The number $d_0$ is independent of the points $b_i$.
\end{enumerate}
 Since $\mcN_{{\sigma}_{N-1}}\left(-N\left( \sum b_i \right)\right)=\mcN_{\tilde{\sigma}_{N-1}}$, the calculation of \cite[Lemma 26]{HT06}, combined with Lemma \ref{lem:preind2}, shows that 
as long as $m$ is at least $\max(m_0, C_{N-1, L})$, we can assemble a new comb over $\FF_{q^m}$, still denoted by $T_{J}$ with a little abuse of notations. Let $p_i\in T_J$ be  points mapped to $r_i$.  The new comb $T_J$ is constructed  by attaching $r$ very free immersed (or embedded if the fiber dimension is at least $3$) rational curves with bounded $\wtH$-degree ($\leq d_0$) in fibers outside $b_1, \ldots, b_l$, such that $\mcN_{T_{J}}\left( -\sum p_i \right)$ is globally generated and has no $H^1$. 
The $\wtH$-degree of all $T_J$ is bounded from above by $d_{N, L} := d_{N-1,L}+(r+ L) d_0$. \\
\begin{figure}[!ht]
 \centering
\vspace*{-0.3in}
\hspace*{-.6in}
 \includegraphics[width=7in]{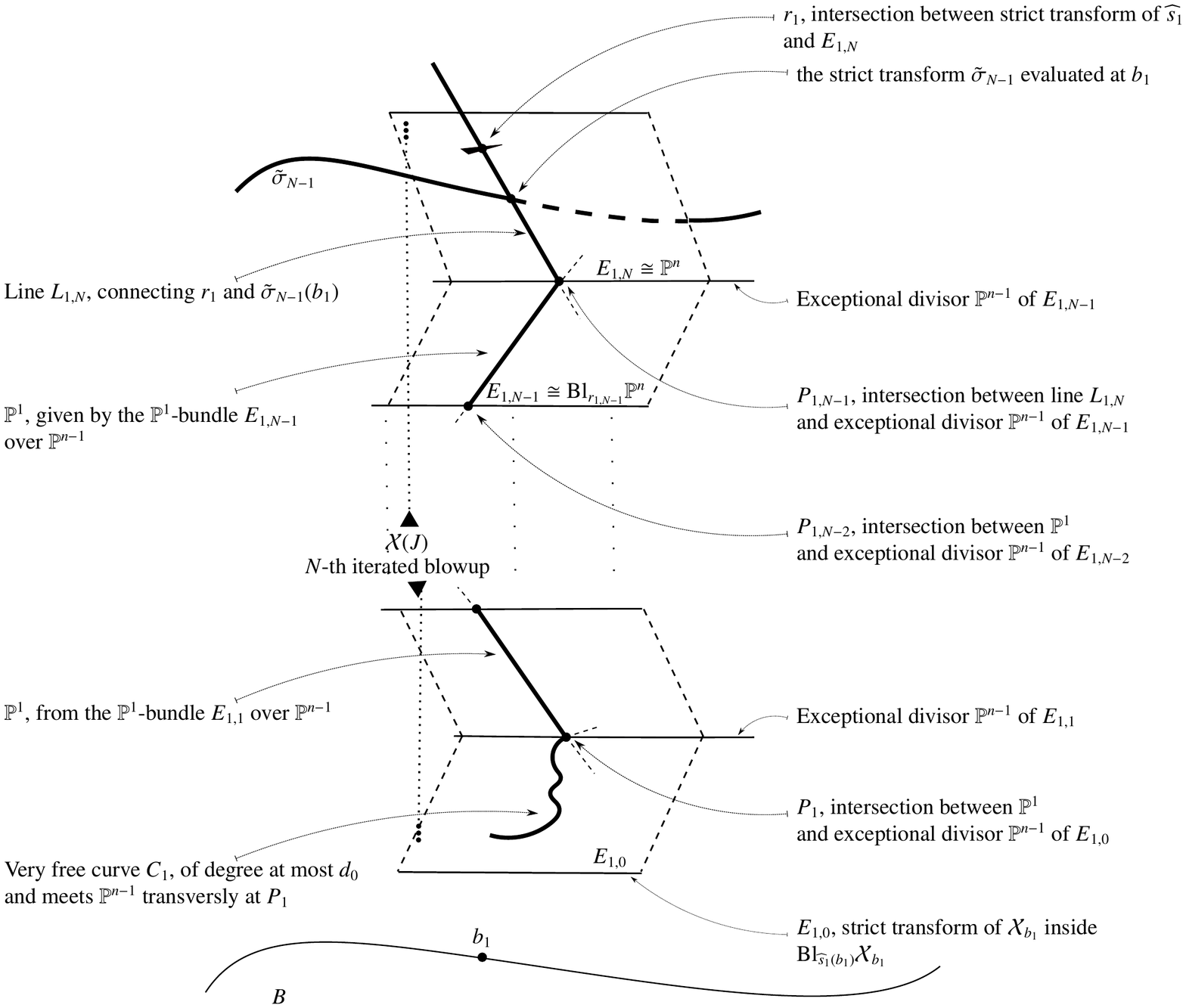}
\vspace*{-.1in}
 \caption{Consider the fiber $\mathcal{X}(J)_{b_1}$ of the $N$-th iterated blow of $\mathcal{X}$ (Figure \ref{fig:comb}), the solid chain of curves -- with $C_1$ at one end,  the line $L_{1,N}$ at the other end --  is a teeth of bounded degree of $T_J$} 
 \label{fig:teeth} 
\end{figure}

Now let us consider a  base  space $\mcB$ parameterizing the following data: 
\[
\left\{
\begin{array}{l}      
	\textrm{a collection of }L\textrm{ distinct points }  (b_1,\ldots,b_L)\in \underbrace{B \times \ldots \times B}_{L};\\
	\textrm{a sequence } (\whts_1/\mfm_{B,b_1}^{N+1},\ldots, \whts_L/\mfm_{B,b_L}^{N+1}) \textrm{ of formal sections $\whts_i$}; \\
	\textrm{iterated blow-up } \mcX(J),~J = (N, (b_1,\ldots,b_L),(\whts_1,\ldots,\whts_L));\\
	\textrm{sequence of points }(r_1,\ldots,r_L)\in ( E_{1,N},\ldots, E_{L,N}),\\
	\textrm{where }r_i\textrm{ is the intersection of strict transform of }\whts_i\textrm{ and } E_{i,N};\\
\end{array}
\right\}, 
\]
then let $\mcM(J)$ be a family over $\mcB$, 
 such that 
$\mcM(J)_{\bb}$ over each point $\bb\in\mcB$ parameterizing the collection of all sections $\sigma:B\to\mcX(J)$ passing $(r_1,\ldots,r_L)$ such that {$\beta(J)\circ\sigma$ is in the smooth locus of $\pi$} and  $\sigma$ has $\wtH$-degree at most $d_{N,L}$ together with an immersion $f$ to $\mcX(J)$. 
Each moduli space $\mcM(J)_{\bb}$ has finitely many components. Consider the closure  $\overline{\mcM}(J)_{\bb}$  of $\mcM(J)_{\bb}$ in the moduli space of stable maps. Let $\overline{M}(J)_\bb$ denote the coarse moduli space of   $\overline{\mcM}(J)_{\bb}$. 

For any $m\geq \max(m_0, C_{N-1,L})$, and any $\FF_{q^m}$-point $\bb\in\mcB$, our construction of comb $T_J$ gives a smooth and non-stacky point of $\overline{\mcM}(J)_{\bb}$. This indicates that there is always a geometric point in $\overline{\mcM}(J)_{\bb}$, and thus there is a geometrically irreducible component in $\overline{\mcM}(J)_{\bb}$. A geometric point of the geometrically irreducible component corresponds to a section $\sigma:B\to\mcX(J)$ passing through $r_1,\ldots,r_L$. Now we are ready to apply Lang-Weil estimate and the strategy is the same as Lemma \ref{lem:preind2}. Mainly, as we varying $\bb\in\mcB$,  $\overline{M}(J)_{\bb}$  forms a family over a finite type scheme. In this family, each fiber has bounded dimension, bounded degree, and bounded $\deg\left(\partial\overline{M}(J)_{\bb}\right)$. All fibers share a universal bound as they vary in a family.  
Then Lang-Weil estimate indicates that there exists $C_{N,L}>0$ such that when $m\geq C_{N,L}$, there is a $\FF_{q^m}$-point in the interior of $\overline{M}(J)_{\bb}$ and hence in the interior of $\overline{\mcM}(J)_{\bb}$. 

Finally, we take the family $\mcM_{N,L} $of sections of $\mcX^\sm$ to be the family of sections whose degree is at most $d_{N, L}$ and agreeing with $\whts_i$ modulo $\mfm_{B,b_i}^{N+1}$ for $i = 1,\ldots,l$.
\end{proof}

In the following we show that we can go back to a smaller field using the special geometry of cubic hypersurfaces. The first is a result to ``lift" formal sections.
\begin{lem}\label{lem:line}
Given a finite sequence $(b_1,\ldots,b_k)$ of distinct closed points of $B$ such that the fiber $\mcX_{b_i}$ is a smooth cubic hypersurface, 
and a sequence $(\whts_1,\ldots,\whts_k)$ of formal power series sections of $\pi$ over $\SP \whtOO_{B, b_i}$, 
there is a sequence $(\whts_{1}{'},\ldots,\whts_{k}{'})$ of formal power series sections of $\pi$ over $\SP \whtOO_{B, b_i}\otimes \FF_{q^2}$ but not over $\SP\whtOO_{B, b_i}$, and a sequence $(\wht{L}_1, \ldots, \wht{L}_k)$ of lines defined over $\SP\whtOO_{B, b_i}$ such that $\wht{L}_i \cap \mcX=\{\whts_i, \whts_i^{'}, \whts_i^{''}\}$, where $\whts_i^{''}$ is the conjugate of $\whts_i^{'}$ under $\text{Gal}(\FF_{q^2}/\FF_q)$.
\end{lem}
\begin{proof}
Let $x_i$ be the intersection of $\whts_i$ with $\mcX_{b_i}$.

By Lemma 9.4 \cite{KollarCubic}, given any smooth cubic hypersurface $X\subset \PP^n, n\geq 2$ over $\FF_q, q\geq 11$, and any $\FF_q$-point $x_i \in X(\FF_q)$, there is a point $y_i \in X(\FF_{q^2})-X(\FF_q)$ and a line $L_i$ defined over $\FF_q$ such that $L_i \cap X=\{x_i, y_i, y_i^\sigma\}$, where $y_i^\sigma$ is the conjugate point of $y_i$ under $\text{Gal}(\FF_{q^2}/\FF_q)$.

So there are lines $\wht{L}_i$ defined over $\SP \whtOO_{B, b_i}$ such that $\wht{L}_i\times \SP\kappa(b_i)=L_i$ and such that $\whts_i \in \wht{L}_i \cap \mcX$. Take $\whts_i^{'}$ and $\whts_i^{''}$ to be the intersection of $\wht{L}_i$ with $\mcX$.
\end{proof}

Now we prove the key reduction lemma.
\begin{lem}\label{lem:reduction}
Let $X$ be a smooth cubic hypersurface defined over $\FF_{q^k}(B)$. If weak approximation at order $N$ holds for  places of good reduction over $\FF_{q^{2k}}(B)$, then weak approximation at order $N$ holds for  places of good reduction over $\FF_{q^k}(B)$.
\end{lem}
\begin{proof}
By Lemma \ref{lem:line}, given a finite sequence $(b_1,\ldots,b_l)$ of distinct closed points of $B$ such that the fiber $\mcX_{b_i}$ is smooth, 
and a sequence $(\whts_1,\ldots,\whts_l)$ of formal power series sections of $\pi$ over $\SP \whtOO_{B, b_i}$, 
there is a sequence $(\whts_1',\ldots,\whts_l')$ of formal power series sections of $\pi$ over $\SP\whtOO_{B, b_i}\otimes \FF_{q^{2k}}$ 
but not over $\SP\whtOO_{B, b_i}$, 
and a sequence $(\wht{L}_1, \ldots, \wht{L}_l)$ of lines defined over $\SP\whtOO_{B, b_i}$ such that $\wht{L}_i \cap \mcX=\{ \whts_i, \whts_i', \whts_i''\}$, where $\whts_i''$ is the conjugate of $\whts_i'$ under $\text{Gal}(\FF_{q^{2k}}/\FF_{q^k})$.

By assumption, there is a section $s'$ defined over $\FF_{q^{2k}}$ such that 
\[
s' \cong \whts_i' \otimes \FF_{q^{2k}}\mod \mfm_{B, b_i}^{N+1}.
\]
By the choice of $\whts_i^{'}$, $s'$ is not defined over $\FF_{q^2}$. Let $s''$ be the conjugate section of $s'$. Then it satisfies
\[
s'' \cong \whts_i'' \otimes \FF_{q^{2k}}\mod \mfm_{B, b_i}^{N+1}.
\]

Take the the family of line $\mathcal{L}$ spanned by $s'$ and $s''$. Note that the family $\mathcal{L}$ is necessarily defined over $\FF_{q^k}$. Let $s$ be the third section in the intersection of $\mathcal{L}$ and $\mcX$. Then this section $s$ meets the approximation requirement by construction.
\end{proof}

We now finish the proof of Theorem \ref{thm:good}. 
For any cubic hypersurface defined over $\FF_q(B)$, there is a model $\pi: \mcX \to B$ over $B$ such that for each closed point $b \in B$, if the corresponding place is a place of good reduction, then the fiber of $\mcX$ over $b$ is a smooth cubic hypersurface.  
By using  general N\'{e}ron desingularization  \cite[Corollary 2.4]{Pop86}, 
we may assume that { for the model $\pi:\mcX\to B$, there is a section $B\to\mcX^\sm$ in its smooth locus}. 
Note that any smooth cubic hypersurface is separably rationally connected. 
Thus the set $S \subset B$ in Lemma \ref{lem:key} is just the set of points over which the fibers are singular. 
Given the collection of closed points $b_1, \ldots, b_l$ and the order $N\geq 0$ we want to approximate, 
there is a number $M$ such that weak approximation at order $N$ holds over these points over $\FF_{q^{2^m}}$ for any $m\geq M$ by the Lemma \ref{lem:key}. Hence, 
 weak approximation at order $N$ for $q\geq 11$ at these points of good reduction follows from Lemma \ref{lem:reduction}.


\section{Proof of Theorem \ref{thm:del}}

In this section, we will discuss  weak approximation for del Pezzo surfaces of degree $4$ at places of good reduction whose residue fields have at least $13$ elements (assuming the existence of rational points) by applying similar ideas from the case of cubic hypersurfaces (cf. Section \ref{sec:cubic}). 

By Lemma \ref{lem:key}, the set of rational points of a del Pezzo surface of degree at least $4$ (in fact any smooth projective separably rationally connected variety) defined over a global function field is either Zariski dense or empty. 
The classical geometry of degree $4$ del Pezzo surfaces naturally suggests that one may blow up a point not contained in the $16$ lines and get a smooth cubic surface with a line, and thus one may reduce the weak approximation problem to the case of cubic surfaces.
However, in our case we cannot use this strategy, at least not in this simple-minded way. 
The reason is the following: for a place $\nu$ of good reduction, we do not know if there is rational point over the function field so that after the blow-up this is a place of good reduction for the cubic surface under our definition (cf. Section \ref{sec:intro}).
More precisely, it is possible that the central fiber of the family only becomes a weak Fano surface and the linear system of the anti-canonical bundle contracts some $(-2)$-curves.
This could happen if all the $\FF_q$-rational points of the fiber are contained in one of the $16$ lines.

Hence, instead of using the reduction to cubic surfaces as the above, we will prove the following lemma, and then weak approximation for del Pezzo surfaces of degree $4$ at places  of good reductions follows by the same argument as in Lemma  \ref{lem:key}, \ref{lem:line} and \ref{lem:reduction} in Section \ref{sec:cubic}.

Throughout the rest of this section, let $X$ be a smooth del Pezzo surface of degree $4$ defined over a finite field $\FF_q$, and the embedding $X \hookrightarrow \PP^4$ is given by the anti-canonical bundle of $X$.
\begin{prop}\label{lem:plane}
 If $q \geq 13$, then given any smooth $\FF_q$-rational point $x$ of $X$, there exists a plane $\Pi$ defined over $\FF_q$ such that the intersection of $\Pi$ with $X$ consists of $x$ and three other points, each of which is defined over $\FF_{q^3}$.
\end{prop}

We need the following lemmas.
\begin{lem}\label{lem:hyperplane}
If $q \geq 5$, then given any smooth $\FF_q$-rational point $x$ of $X$, there exists a hyperplane $H\subset\PP^4$ containing $x$ defined over $\FF_q$ such that the intersection of $H$ with $X$ is a geometrically integral genus $1$ curve which is smooth at $x$.
\end{lem}
\begin{proof}
The intersection $H \cap X$ fails to be geometrically integral if and only if $H$ contains a line or a conic. 
In the case where $H$ contains a conic not passing through $x$, $H$ will contain a ``residual" conic (or a chain of two lines) which contains the point $x$.  

The geometry of del Pezzo surfaces of degree $4$ provides us with the following information. 
\begin{itemize}
	\item  Since there are $16$ lines in $X$, we have at most $16(q+1)$ hyperplanes containing a line.
	\item Since there are at most $5$ conics containing $x$, and any hyperplane containing a conic also contains the plane spanned by the conic, we have at most $5(q+1)$ hyperplanes containing a conic.
	\item There are $q+1$ tangent hyperplanes at $x$.
\end{itemize}
There are $q^3+q^2+q+1=(q^2+1)(q+1)$ lines containing $x$. As long as $q \geq 5$, we can find a hyperplane over $\FF_q$ with desired properties.
\end{proof}

\begin{lem}\label{lem:lineInPlane}
Let $E \subset \PP^2$ be a geometrically integral degree $3$ curve over a finite field $\FF_q$ . If $q \geq 13$, then given any smooth $\FF_q$-rational point $x$ of $X$, there exists a line $L$ defined over $\FF_q$ such that the intersection of $L$ with $X$ consists three Galois conjugate points, all of which are defined over $\FF_{q^3}$.
\end{lem}
\begin{proof}
We first consider the case $E$ is smooth.

Denote by $n$ (resp. $m$) the number of $\FF_q$ (resp. $\FF_{q^2}$) points of $E$. There are three types of lines we want to avoid.
\begin{itemize}
	\item There are $n$ tangent lines to $E$.
	\item There are  at most $\frac{m-n}{2}$ lines which intersect $E$ at one $\FF_q$ point and two $\FF_{q^2}$ (but not $\FF_q$) points.
	\item There are at most $\frac{1}{3}{n\choose 2}$ lines which intersect $E$ at three $\FF_q$ points. 
\end{itemize}

Thus there are at most $\frac{n(n+2)}{6}+\frac{m}{2}$ lines to avoid. 
By the Hasse-Weil estimate, $n\leq 1+q+2\sqrt{q}, m \leq 1+q^2+2q$.
So when $q\geq 13$, there is at least one line which is not of the three types listed above.

If $E$ is singular, then it is either a nodal or a cupidal cubic. In either case there is a unique singular point.
Assume the number of $\FF_q$ (resp. $\FF_{q^2}$) points of $E$ is $1+n$ (resp. $1+m$).
Then $n=q-1, q$ or $q+1$, $m=q^2$ or $q^2-1$.

There are $q+1$ lines containing the unique singular point,
$n$ lines tangent to $E$ at an $\FF_q$ point,
$\frac{m-n}{2}$ lines which intersect $E$ at one $\FF_q$ point and two $\FF_{q^2}$ (but not $\FF_q$) points,
at most $\frac{1}{3}{n\choose 2}$ lines which intersect $E$ at three $\FF_q$ points.
By a simple computation, for any $q$ we can find the desired line $L$. 
\end{proof}

Once these are proved, Proposition \ref{lem:plane} follows easily.
\begin{proof}[Proof of Proposition \ref{lem:plane}]
We use the hyperplane $H\subset\PP^4$ given by Lemma \ref{lem:hyperplane}, and take the intersection $H \cap X$  to get a genus one curve in $H \cong \PP^3$.
Projecting the curve from $x$ yields a degree $3$ curve in $\PP^2\subset H\subset\PP^4$.
Take the plane in $H$ spanned of $x$ and the line given by Lemma \ref{lem:lineInPlane}. This is the plane we want.
\end{proof}


\section{Proof of Theorem \ref{thm:function}}
The general idea of the proof is first to approximate the $\nu$-adic points by a rational point $p$ in a degree $2$ field extension $K'$ of $K$, 
then produce a rational curve in the original field $K$ containing all the Galois conjugate points of $p$. 
Since weak approximation holds for $\PP^1_K$, we prove weak approximation over the original field $K$.
This strategy depends closely on the existence of rational curves connecting two general rational points over a global field, which we discuss below.

\begin{lem}\label{lem:singular}
Let $K$ be a purely imaginary number field or the function field of a smooth irreducible curve $B$  defined over a finite field. 
Let $V$ be a singular cubic hypersurface of dimension at least $4$ with an ordinary double point $P$ which is a $K$-rational point.
Then given any other $K$-rational point $R$, there is a $K$-morphism $f: \PP^1_K \to V$ such that $f(0)=P, f(\infty)=R$. 
\end{lem}
\begin{proof}
We may assume that the line $\overline{PR}$ spanned by the two points $P$ and $R$ is not contained in $V$, otherwise there is nothing to prove. 

Without loss of generality, we can assume $P$ is the point $[1, 0, \ldots, 0]\in\PP^{n}_{[X_0,\ldots,X_n]}$. 
Then we can write the equation of the cubic hypersurface $V$ as $X_0 Q(X_1, \ldots, X_n)+C(X_1, \ldots, X_n)=0$.
Denote the coordinate of $R$ by $[r_0, r_1, \ldots, r_n]$. 
Since the line $\overline{PR}$ is not in $V$, $Q(r_1, \ldots, r_n)\neq 0$.
In particular at least one of $r_1, \ldots, r_n$ is non-zero. 

Since the point $P$ is an ordinary double point, the equation $Q(X_1, \ldots, X_n)=0$ defines a smooth quadric hypersurface of dimension at least $3$ in $\PP^{n-1}$.
By the assumption on the field $K$, the set of rational points of this smooth quadric hypersurface is Zariski dense.
In particular, there is a point $S=[s_1, \ldots, s_n]$ in $\PP^{n-1}$ such that $Q(s_1, \ldots, s_n)=0$ and $C(s_1, \ldots, s_n)$ does not vanish.

Now consider the $k$-rational map
\begin{align*}
\phi: \PP^{n-1} &\dashrightarrow V\\
 [X_1, \ldots, X_n] &\mapsto [-C(X_1, \ldots, X_n), X_1Q(X_1, \ldots, X_n), \ldots, X_nQ(X_1, \ldots, X_n)]\\
\end{align*}
This gives a birational isomorphism between $\PP^{n-1}$ and $V$. 
Furthermore, the birational map $\phi$ is a morphism near the point $S$ and $[r_1, \ldots, r_n]$ by the above discussion. 
In particular,  $\phi(S)=P, \phi([r_1, \ldots, r_n])=R$. 
Then $\phi$ maps the line in $\PP^{n-1}$ connecting $S$ and $[r_1, \ldots, r_n]$ to a rational curve connecting $P$ and $R$ in $V$. 
\end{proof}

 The following lemma is essentially proved in \cite[Proposition 1.4]{MadoreREquivalence}. 
Our version has stronger hypothesis and stronger result. 
\begin{lem}\label{lem:Rconnected}
Let $K$ be either a purely imaginary number field or a global function field and let $X$ be a smooth cubic hypersurface of dimension $n\geq 10$ defined over $K$. For any two $K$-rational points $x$ and $y$, such that the line $L$ spanned by $x$ and $y$ intersects $X$ at three distinct points $x, y, z$,
denote by $H_x, H_y, H_z$ the projective tangent hyperplane at $x, y, z$. Assume the followings hold.
\begin{enumerate}
\item The family of lines passing through $x$ is geometrically irreducible and has dimension $n-3$.
\item The intersection $H_z \cap X$ is a singular cubic hypersurface with an ordinary double point $z$. 
\item The intersection $H_x \cap H_z \cap X$ is a smooth cubic hypersurface of dimension $n-2$.
\end{enumerate}
 Then there is a $K$-morphism $f: \PP^1_K \to X$ such that $f(0)=x, f(\infty)=y$. 
\end{lem}
\begin{proof}
By assumption neither $x$ nor $y$ is contained in $H_z$.

Let $t_x: X \dashrightarrow X$ be the birational involution defined by sending a general point $p$ to the point $q$ such that $x, p, q$ are colinear. 
For a point $p$ in $H_x \cap X$, not equal to the point $x$, such that the line spanned by $p$ and $x$ is not contained in $X$, 
the birational map $t_x$ is defined near $p$ and $t_x(p)=x$.

The intersection $H_x \cap H_z \cap X$ is a divisor in $H_x \cap X$. 
The locus of points swept out by the family of lines through $x$,
\[
\text{Line}(x)=\{w| w\neq x, \text{ and the line spanned by } x \text{ and } w \text{ is contained in } X\} \cup \{x\}
\]
is an irreducible divisor and contains the point $x$ by definition. 
Since $H_z$ does not contain $x$, the intersection $H_x \cap H_z \cap X$ is not contained in $\text{Line}(x)$.
By the assumption on the field $K$, there is a $K$-rational point $u$ in the intersection $H_x \cap H_z \cap Z$.
This is clear if $K$ is a global function field. 
For the case of number fields, this is proved in \cite{BrowningVishe}.
By \cite[Thereom 1]{KollarCubicUnirational}, the intersection is then unirational and thus has a Zariski dense set of rational points.
In particular, there is a $K$-rational point $u$ such that the birational involution $t_x$ is defined around $u$.  
Lemma \ref{lem:singular} shows that there is rational curve in $H_z \cap X$, thus also in $X$, containing the two points $z$ and $u$. 
Then the birational map $t_x$ maps the rational curve to a rational curve containing $x$ and $y$.
\end{proof}

\begin{thm}\label{thm:WAhighdim}
Let $X$ be a smooth cubic hypersurface in $\PP^n, n\geq 11$ defined over a global field $K$, which is either a purely imaginary field or $\FF_q(B)$, the function field of a curve $B$ defined over a finite field $\FF_q$ of characteristic not equal to $2, 3, 5$. Then $X$ satisfies weak approximation.
\end{thm}
\begin{proof}
Under the assumptions, there exist $K$-rational points of $X$. 
This is clear if $K$ is a global function field. 
For the case of number fields, this is proved in \cite{BrowningVishe}.
Then by \cite[Theorem 1]{KollarCubicUnirational}, the intersection is unirational and thus has a Zariski dense set of rational points.
Choose $x$ to be a general point such that
\begin{itemize}
\item
The family of lines through $x$ is geometrically irreducible and has dimension $n-3$.
\item
Denote by $H_x$ the projective tangent hyperplane at $x$. Then the projective tangent cone of $H_x \cap X$ has an ordinary double point at $x$.
\end{itemize}
These two properties can be achieved since the set of points in $X$ satisfying these properties is non-empty and Zariski dense (\cite[Lemma 5.4, 5.7]{Hasse}).

Denote by $x_\nu$ the $K_\nu$-rational point induced by $x$ for each place $\nu$. 
Given finitely many places $\nu_1, \ldots, \nu_k$ and any $\nu_i$-adic point $y_{\nu_i}$ which we want to approximate, let $M_i$ be the line spanned by $x_{\nu_i}$ and $y_{\nu_i}$. 
Without loss of generality, we may assume that the line $M_i$ intersect the cubic hypersurface at a third point $z_{\nu_i}$ different from $x_{\nu_i}$ and $y_{\nu_i}$. 
We may also assume that the points $z_{\nu_i}$ are general points such that
\begin{itemize}
\item The family of lines through $z_{\nu_i}$ is geometrically irreducible and has dimension $n-3$.
\item Let $H_{z_{\nu_i}}$ be tangent hyperplane at $z_{\nu_i}$. Then $H_{z_{\nu_i}} \cap X$ is a singular cubic hypersurface with an ordinary double point $z_{\nu_i}$.
\item The intersection $H_{x_{\nu_i}} \cap H_{z_{\nu_i}} \cap X$ is a smooth cubic hypersurface of dimension $n-2$.
\end{itemize}

For any pre-specified order of weak approximation, there is a line $M$ defined over $K$ which approximates all the $M_i$'s to that order and contains the point $x$. 
This follows from weak approximation for the space of lines $M\subset X$ containing $x$, 
which holds as this space is isomorphic to $\PP^d_K$ for some $d$.

If the line $M$ intersect $X$ at two $K$-rational points $y, z$, then by Lemma \ref{lem:Rconnected}, there is a rational curve containing them (note that our assumptions on $z_{\nu_i}$ imply the assumptions of the lemma). 
In particular, the base change of this rational curve to $K_{\nu_i}$ contains $K_{\nu_i}$ points which approximate the points $y_{\nu_i}$ to order $N$. 
Then the theorem follows since we know weak approximation holds for $\PP^1_K$ and thus can find a rational point in $\PP^1_K$ which approximates the $\nu_i$-adic points.

Now assume the line $M$ intersect $X$ at two points $y, z$ which are defined over a separable degree $2$ field extension $L/K$. 
By construction, the irreducible (over $K$) degree $2$ zero cycle $y+z$ is close to the degree $2$ zero cycle $y_{\nu_i}+z_{\nu_i}$ in the $\nu_i$-adic topology for each $\nu_i$.

In the following we will show that there is a rational curve defined over $K$ which contains the degree $2$ zero cycle $y+z$. Once we know this, the theorem follows for the same reason as above.

Consider the Weil restriction of scalars $\text{Res}_{L/K}X_L$ and the $K$-rational dominant map $\text{Res}_{L/K}X_L \dashrightarrow X$. After base change to $L$, the Weil restriction is isomorphic to $X_L \times X_L$ and the rational map maps a general pair of points $(u, v)$ to $w$ such that $u, v, w$ are colinear (c.f. \cite[Section 15, Proposition 15.1]{ManinCubic}, \cite[Example 3.8, Exercise 3.11]{KollarCubic},  and a detailed discussion of the former two in \cite[Construction 2.1]{Requiv}).

Recall that by the universal property of Weil restrictions, a $K$-morphism $f: \PP^1_K \to \text{Res}_{L/K}X_L$ is the same as an $L$-morphism $g: \PP^1_L \to X_L$; and an $L$-point of $\text{Res}_{L/K}X_L$ is the same as two $L$-points of $X_L$.

There are two $L$-rational points $u, v$ of $Res_{L/K}X_L$, 
which are mapped to $z$ and $y$ respectively, i.e. the points corresponding to the pair of $L$-rational points $(x, y)$ and $(z, x)$.
Moreover these two points are conjugate to each other by the Galois group $\text{Gal}(L/K)$.

So it suffices to show that there is a rational curve over $K$ which contains one of the two Galois conjugate $L$-rational points $u, v$ of $Res_{L/K}X_L$.
This is equivalent of showing that there is a rational curve defined over $L$ containing $x$ and $y$ (or a rational curve defined over $L$ containing $x$ and $z$), which follows from Lemma \ref{lem:Rconnected}.
\end{proof}

\bibliographystyle{alpha}
\bibliography{MyBib}

\end{document}